\documentclass[reqno]{amsart}
\usepackage{graphicx, amssymb, color,pdfsync}
\usepackage[all,cmtip]{xy}
\numberwithin{equation}{section}
\usepackage{mathrsfs}

\usepackage{tikz}
\usetikzlibrary{matrix,arrows}

\usepackage{hyperref}
\hypersetup{
  colorlinks   = true, 
  urlcolor     = blue, 
  linkcolor    = blue, 
  citecolor   = red 
}

\newcommand{\Stab}{\mathop{\rm Stab}\nolimits} 

\newtheorem{theo}{Theorem} 
\newtheorem{prop}[theo]{Proposition}

\newtheorem{lemm}[theo]{Lemma}

\theoremstyle{remark}
\newtheorem{rema}[theo]{\bf Remark}

\title[A sufficiently complicated noded Schottky group of rank three]{A sufficiently complicated noded Schottky group of rank three}

\author{Rub\'en A. Hidalgo} 
\thanks{Partially supported by Projects Fondecyt 1150003 and Anillo ACT1415 PIA-CONICYT} 
\keywords{Riemann surfaces; Schottky groups} 
\subjclass[2010]{30F40, 30F10}
\address{Departamento de Matem\'atica y Estad\'{\i}stica, Universidad de La Frontera, Temuco, Chile} 
\email{ruben.hidalgo@ufrontera.cl}

\begin{document}

\begin{abstract}
The theoretical existence of non-classical Schottky groups is due to Marden and explicit examples are only known in rank two; the first one by Yamamoto in 1991 and later by Williams in 2009. In 2006, Maskit and the author provided a theoretical method to construct non-classical Schottky groups in any rank. The method assumes the knowledge of certain algebraic limits of Schottky groups, called sufficiently complicated noded Schottky groups. We provide explicitly a sufficiently complicated noded Schottky group of rank three and explain how to use it to construct explicit non-classical Schottky groups. 

\end{abstract}

\maketitle

A {\it Schottky group of rank $g \geq 1$} is a group $G$ generated by (necessarily loxodromic) M\"obius transformations  $A_{1}$,..., $A_{g}$, so that there is  a collection of $2g$ pairwise disjoint simple loops $\alpha_{1}, \alpha'_{1}, \alpha_{2}, \alpha'_{2}, \ldots,\alpha_{g}, \alpha'_{g}$ on the Riemann sphere $\widehat{\mathbb C}$, all of them bounding a common domain ${\mathcal D}$ of connectivity $2g$, with $A_{j}(\alpha_{j})=\alpha'_{g}$ and $A_{j}({\mathcal D}) \cap {\mathcal D}=\emptyset$. The above set of generators and set of loops are called, respectively, {\it geometrical generators} and a {\it fundamental set of loops} for $G$. The open set ${\mathcal D}$ turns out to be a fundamental domain for $G$, its region of discontinuity $\Omega$ is a connected set  and $\Omega/G$ is a closed Riemann surface of genus $g$. Koebe's uniformization theorem asserts that every closed Riemann surface can be obtained (up to conformal isomorphisms) in that way. Chuckrow \cite{Chuckrow} proved that every set of $g$ generators of a Schottky group of rank $g$ is geometrical and Maskit \cite{Maskit:Schottky} proved that a Schottky group of rank $g$ may equivalently be defined as a purely loxodromic Kleinian group with non-empty region of discontinuity and isomorphic to a free group of rank $g$. 

Schottky groups were firstly considered by Schottky around 1882  by using as a fundamental set of loops a collection of Euclidean circles; these are now called {\it classical Schottky groups} and the corresponding set of generators a {\it classical set of generators}.  In general, a classical Schottky group may have non-classical set of generators. Examples of classical Schottky groups are, for instance, finitely generated purely hyperbolic Fuchsian groups representing a closed surface with holes \cite{Button} and, in the special case the Fuchsian group is a two generator group representing a torus with one hole, this is classical on every set of generators \cite{P}. 

A conjecture (it seems due to Bers) is that every closed Riemann surface can be obtained from a suitable classical Schottky group. Some positive answers were obtained by Bobenko \cite{Bobenko}, Koebe \cite{Koebe}, Maskit \cite{Maskit:conj}, Sepp\"al\"a \cite{Seppala} (for the case the surface admits antiholomorphic involutions with fixed points) and McMullen \cite{McMullen} (in the case the surface has sufficiently shorts geodesics). Recently, Hou \cite{Hou1,Hou2} have announced a proof of this conjecture (by using Haussdorf dimension of the limit set of Kleinian groups) and another approach in \cite{Hidalgo:classical} (by using Belyi curves).

In 1974, Marden \cite{Marden} provided a theoretical (non-constructive) existence of non-classical Schottky groups. In 1975, Zarrow \cite{Zarrow} claimed to have constructed an example of a non-classical Schottky group of rank two, but it was lately noted to be incorrect by Sato in \cite{Sato}. The first explicit (correct) construction was provided by Yamamoto \cite{Yamamoto} in 1991 and later by Williams in his Ph.D. Thesis \cite{Williams} in 2009, both for  rank two. It seems that, for rank at least three, there is not explicit example in the literature. 

In \cite{H-M}, Maskit and the author described a theoretical method to construct non-classical Schottky groups in any rank $g \geq 2$. The idea is to consider certain Kleinian groups, obtained as geometrically finite algebraic limits of Schottky groups of rank $g$, called {\it noded Schottky groups of rank $g$}. These groups can be geometrically defined as follows. Consider a collection of pairwise disjoint open topological discs $D_{1}, D'_{1},\ldots, D_{g}, D'_{g}$ on the Riemann sphere $\widehat{\mathbb C}$ so that the corresponding boundaries $\widehat{\alpha}_{1}=\partial D_{1}, \widehat{\alpha}'_{1}=\partial D'_{1}, \ldots, \widehat{\alpha}_{g}=\partial D_{g}, \widehat{\alpha}'_{g}=\partial D'_{g}$
are simple loops and they only intersect in at most finitely many points. 
Assume $\widehat{A}_{1},\ldots, \widehat{A}_{g}$ are M\"obius transformations so that  $\widehat{A}_{j}(\widehat{\alpha}_{j})=\widehat{\alpha}'_{j}$ and
$\widehat{A}_{j}(D_{j})\cap D'_{j}=\emptyset$, for each $j=1,\ldots,g$. Observe that the transformation $\widehat{A}_{j}$ may only be loxodromic or parabolic. The group $\widehat{G}$, generated by these transformations, is a Kleinian group isomorphic to a free group of rank $g$. If $p$ is a point of intersection of two of the above loops, then either it is a fixed point of a parabolic transformation of $\widehat{G}$ or it has trivial $\widehat{G}$-stabilizer. In the second situation, one may deform in a suitable manner these loops in order to avoid the intersection at $p$ and not adding extra intersections.  In this way, two possibilities appear (up to performing the above deformation), either (i) $G$ is a Schottky group of rank $g$ or (ii) there are intersection points, each of them being a 
 fixed point of a parabolic transformation of $\widehat{G}$. In case (ii) we say that $\widehat{G}$ is a {\it noded Schottky group of rank $g$}; we call the above set of loops {\it a set of defining loops} and the generators a set of {\it geometrical generators}. In \cite{bm:combiv}, as an application of the Klein-Maskit's combination theorems,  it was noted that $\widehat{G}$ is  geometrically finite, that each of its parabolic elements is a conjugate of a power of one of the transformations fixing a common point of two of the defining loops, and that the complement $\widehat{\mathcal D}$ of the union of the closures of $D_{1}$, $D'_{1}$,..., $D_{g}$, $D'_{g}$ is a fundamental domain for $\widehat{G}$. Unfortunately, different as was for the case of Schottky groups, not every set of free generators of a noded Schottky group is necessarily geometrical.  

In \cite{bm:free} is was observed that a geometrically finite Kleinian group, with non-empty region of discontinuity, isomorphic to a free group is either a Schottky group or a  noded Schottky group (that every geometrically finite Kleinian group isomorphic to a free group has non-empty region of discontinuity was observed in \cite{Hidalgo:NodedSchottky}). In particular, a Schottky group (respectively, a noded Schottky group) may equivalently be defined as a geometrically finite Kleinian group isomorphic to a free group without (respectively, containing) parabolic elements.

If $\Omega$ is the region of discontinuity of a noded Schottky group $\widehat{G}$ of rank $g$, then by adding to it the parabolic fixed points of $\widehat G$, and with the appropriate cusped topology (see \cite{Hidalgo:NodedFuchsian,K-M:pinched}), we obtain its {\it extended region of discontinuity} $\Omega^+$; it happens that  $\Omega^+/\widehat G$ is a stable Riemann surface of genus $g$. It was observed in \cite{Hidalgo:NodedSchottky} that every stable Riemann surface of genus $g$ is obtained as above; so every point of the Deligne-Mumford compactification of the moduli space of genus $g$ can be realized by a suitable noded Schottky group of rank $g$. 

A noded Schottky group for which there is a set of geometrical generators admitting a set of defining loops all of which are Euclidean circles is called {\it neoclassical}; the set of generators is a {\it neoclassical set of generators}.  With respect to the classical Schottky uniformization conjecture at the level of stable Riemann surfaces, it was proved in \cite{H-M} that every noded Schottky group of rank three providing a stable Riemann surface as in Figure \ref{figura0} cannot be neoclassical (this should be still true for every noded Schottky group of rank $g \geq 4$ whose corresponding stable Riemann surface has $g+1$ components, one being of genus zero and the others being of genus one). 

In the same paper  it was introduced the concept for a noded Schottky group to be {\it sufficiently complicated}  (we recall such a technical definition in Section \ref{Sec:complicado} as well some results from \cite{H-M}) and it was proved that Schottky groups which are sufficiently near (i.e., in a conical neighbourhood; see Section \ref{Sec:complicado}) to a sufficiently complicated noded Schottky group is necessarily non-classical. In this way, in order to produce infinitely many non-classical Schottky groups of a given rank, we just need to construct explicitly a sufficiently complicated noded Schottky group of such rank. In \cite{H-M} it was seen that a noded Schottky group that uniformizes a stable Riemann surface of genus three as shown in Figure \ref{figura1} is sufficiently complicated (this should still be true for a noded Schottky group of rank $g \geq 4$ whose corresponding stable Riemann surface has one component has genus zero and others $g$ being genus one noded surfaces). In Section \ref{Sec:construccion} we construct such an explicit (sufficiently complicated) noded Schottky group of rank $3$ by finding explicit generators and we explain how to use this to produce infinitely many non-classical Schottky groups of rank three.

\begin{figure}[!tbp]
  \centering
  \begin{minipage}[b]{0.49\textwidth}
  \centering
    \includegraphics[width=3.0cm]{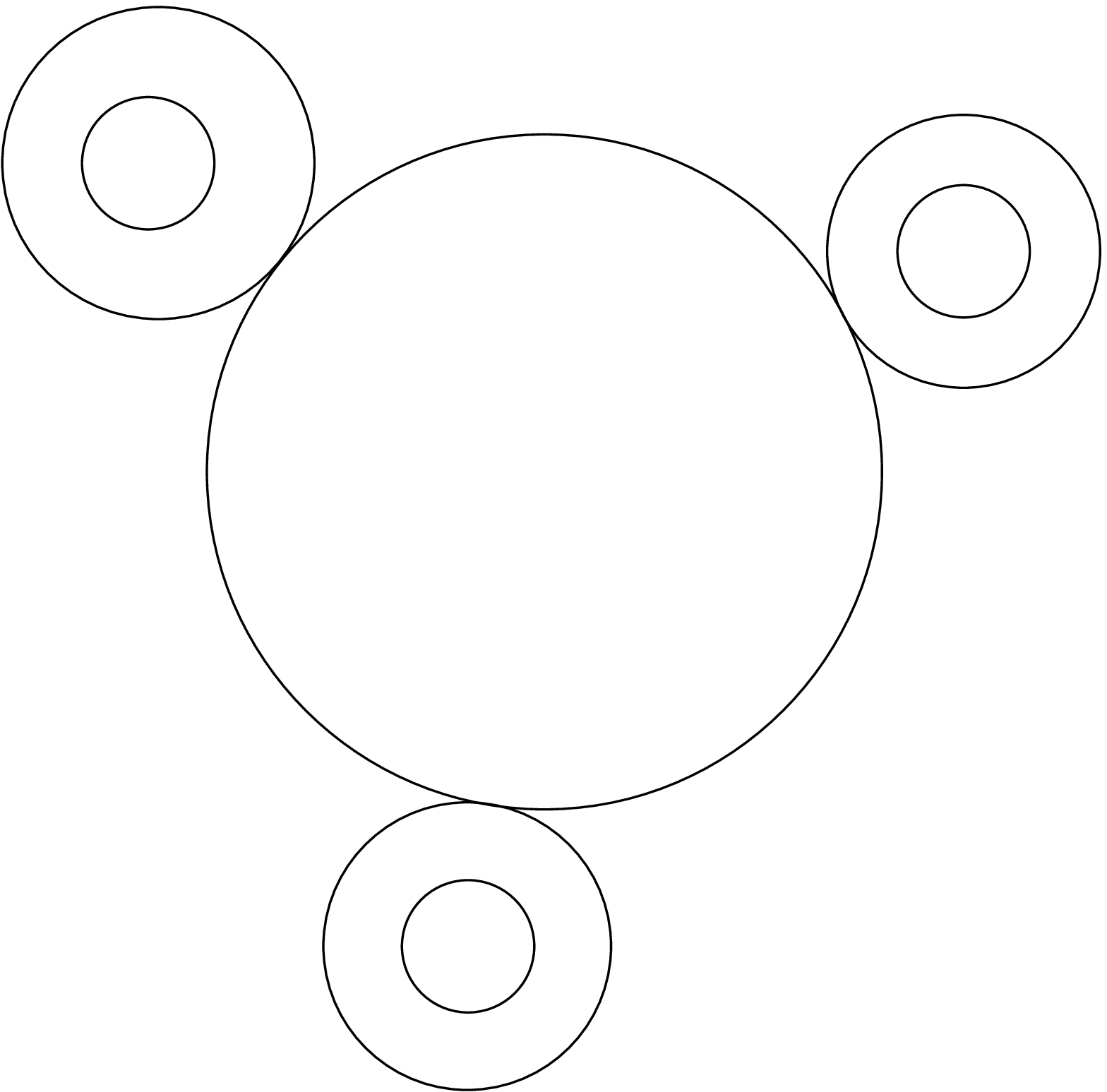}
    \caption{A stable Riemann surface of genus $3$ with $3$ nodes}
    \label{figura0}
  \end{minipage}
  \hfill
  \begin{minipage}[b]{0.49\textwidth}
  \centering
    \includegraphics[width=3.0cm]{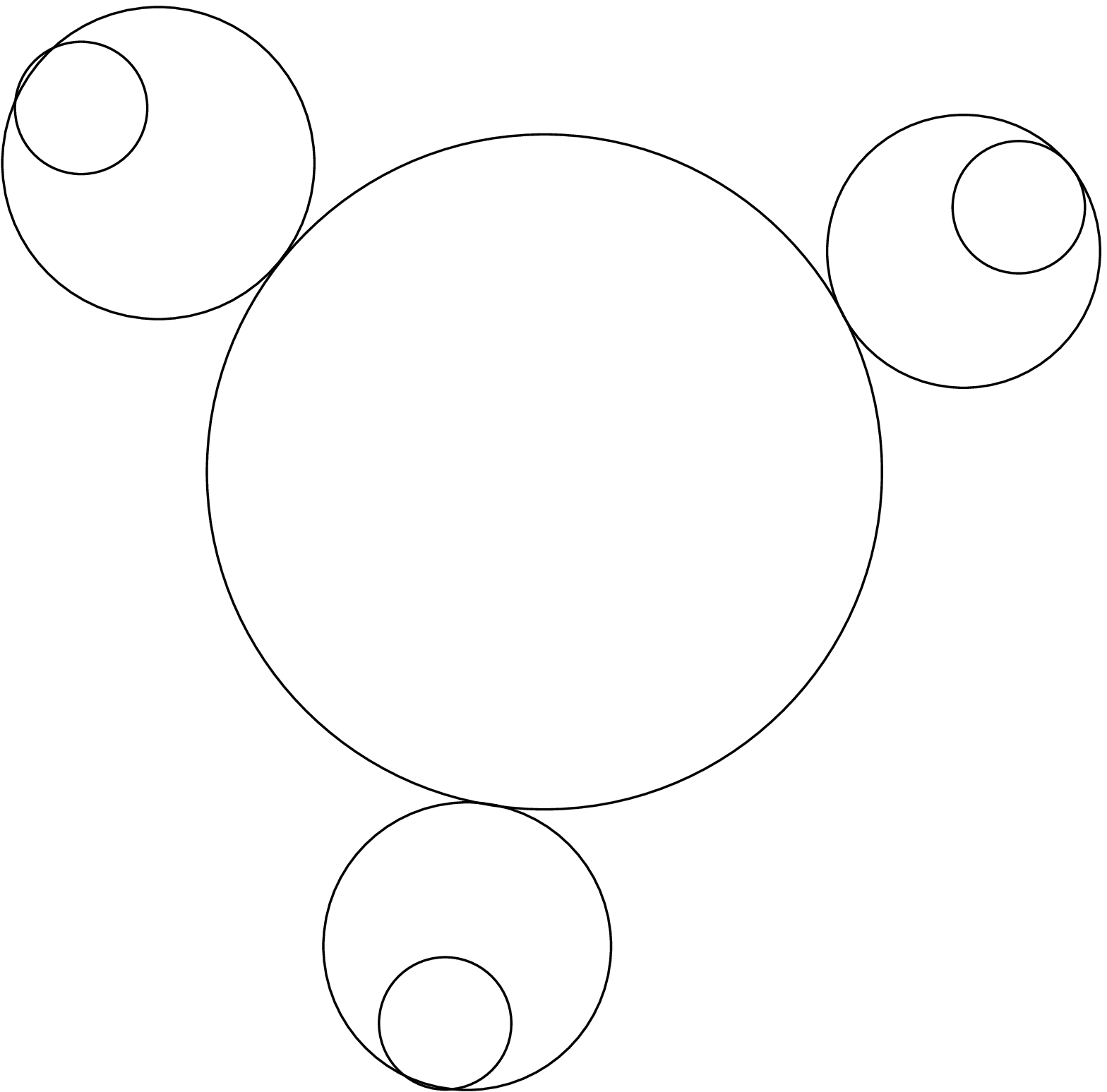}
    \caption{A stable Riemann surface of genus $3$ with $6$ nodes}
    \label{figura1}
  \end{minipage}
\end{figure}

\section{Sufficiently Complicated Noded Schottky Groups} \label{Sec:complicado}
In this section we recall the definition of sufficiently complicated noded Schottky groups and some of results from \cite{H-M}. The space of deformations of a Schottky group of rank $g$, denoted by ${\mathcal S}_{alg}$, is a subset of the representation space of the free group of rank $g$ in ${\rm PSL}_{2}({\mathbb C})$, modulo
conjugation. Regard ${\mathbb H}^{3}$ as being the set $\{(z,t):z\in{\mathbb C},
t>0\in{\mathbb R}\}$. We likewise identify ${\mathbb C}$ with the boundary of ${\mathbb H}^{3}$, except for the point at infinity; that is, we identify ${\mathbb C}$ with $\{(z,t):t=0\}$.

\subsection{The relative conical neighbourhood of a noded Schottky group}\label{sec:vertical}
Let $\widehat{G}$ be a noded Schottky group of rank $g \geq 2$, with a set of geometrical generators $\widehat{A}_{1}, \ldots, \widehat{A}_{g}$, and corresponding fundamental system of loops $\widehat \alpha_1,\ldots,\widehat \alpha'_g$, these being the corresponding boundary loops of 
 a collection of pairwise disjoint open discs $D_{1},\ldots, D'_{g}$.
 The complement of the closures of these discs is a fundamental domain $\widehat D$ for $\widehat G$.  Let ${\widehat P}_1,\ldots,{\widehat P}_{q}$ be a maximal set of primitive parabolic elements of $\widehat G$ generating non-conjugate cyclic subgroups, where $q \geq 1$ (we may assume the fix point of these parabolic transformations are contained in the intersection of two fundamental loops). We denote by $\Omega(\widehat{G})$ its region of discontinuity and by $\Omega^{+}(\widehat{G})$ its extended region of discontinuity. 
Next, we proceed to recall a construction done in \cite{H-M} of a one-real family of Schottky groups $G^{\tau}$ whose geometric limit is $\widehat{G}$.

\subsubsection{The infinite shoebox construction}
For each $i=1,\ldots,q$, choose a particular M\"obius transformation $H_i$ conjugating $\widehat P_i$ to the transformation $P(z)=z+1$ and consider the renormalized group $H_i \widehat G H_i^{-1}$. For this group, there is a number $\tau_0>1$ so that the set $\{|\Im(z)|\ge\tau _0\}$ is precisely invariant under the stabilizer $\Stab(\infty)$ of $\infty$ in the group $H_i \widehat G H_i^{-1}$. In this normalization, for each parameter $\tau$, with $\tau >\tau _0$, we define the {\em infinite shoebox} to be the set $B_{0,\tau}=\{(z,t):|\Im(z)|\le \tau ,t\le\tau \}$. 
Since $\tau_0>1$, we easily observe that for every $\tau>\tau _0$, the complement of $B_{0,\tau}$ in ${\mathbb H}^{3}\cup{\mathbb C}$ is precisely invariant under $\Stab(\infty)\subset H_i \widehat G H_i^{-1}$, where we are now regarding M\"obius transformations as hyperbolic isometries, which act on the closure of ${\mathbb H}^{3}$. Then for $\widehat G$, the infinite shoebox with parameter $\tau$ at $z_i$, the fixed point of $\widehat P_i$, is $B_{i,\tau}=H_i^{-1}(B_{0,\tau})$. If $\widehat P$ is any parabolic element of $\widehat G$, conjugate to some power of $\widehat P_i$, then the corresponding infinite shoebox at the fixed point of $\widehat P$, is given by $T(B_{i,\tau })$, where $\widehat P=T \widehat P_i T^{-1}$.
It was observed in \cite{bm:free} that, for each fixed $\tau>\tau _0$, $\widehat G$ acts as a group of conformal homeomorphisms on the {\em expanded regular set} $B^{\tau}=\bigcap \widehat A(B_{i,\tau})$, where the intersection is taken over all $\widehat A \in \widehat G$ and all $i=1,\ldots,q$. Further, $\widehat G$ acts as a (topological) Schottky group (in the sense of our geometrical definition) on the boundary of $B^{\tau }$. 
Each parabolic $\widehat P \in\widehat G$ appears to have two fixed points on the boundary of $B^{\tau}$; that is, $\widehat P$, as it acts on the boundary of $B^{\tau}$, appears to be loxodromic. The {\em flat part} of $B^{\tau }$ is the intersection of $B^{\tau}$ with the extended complex plane $\widehat{\mathbb C}$. The complement of the flat part (on the boundary of $B^{\tau}$) is the disjoint union of 3-sided {\em boxes}, where each box has two {\em vertical sides}
(translates of the sets $\{\Im(z)=\pm \tau , 0<t<\tau \}$) and one {\em horoball side} (a translate of the set $\{|\Im(z)| \leq \tau , t=\tau \}$).  For each $i=1,...,q$ and for each $n=1,2,...$, we set $B_{i,\tau ,n}=H_{i}^{-1}(B_{0,\tau } \cap \{|\Re(z)|\leq n\})$ and  $B^{\tau,n}=\bigcap_{\widehat{A} \in \widehat{G}} \widehat{A}(B_{i,\tau,n})$.
The {\em truncated flat part} of $B^{\tau,n}$ is the intersection $B^{\tau ,n} \cap \widehat{\mathbb C}$. The boundary of the truncated flat part near a parabolic fixed point, renormalized so as to lie at $\infty$, is a Euclidean rectangle.
Let us renormalize $\widehat G$ so that $\infty \in  \Omega(\widehat G)$. Then, for each $\tau>\tau _0$, there is a conformal map $f^{\tau}$, mapping the boundary of $B^{\tau}$ to $\widehat{\mathbb C}$, and conjugating $\widehat G$ onto a Schottky group $G^{\tau }$, where $f^{\tau }$ is defined by the requirement that, near $\infty$, $f^{\tau}(z)=z+O(|z|^{-1})$. The group $G^{\tau}$ depends on the choice of the M\"obius transformations $H_1,\ldots, H_q$ as well as on the choice of $\tau$. The elements $A_{1}^{\tau}=f^{\tau} \widehat{A}_{1} (f^{\tau})^{-1}, \ldots,  A_{g}^{\tau}=f^{\tau} \widehat{A}_{g} (f^{\tau})^{-1}$ provides a set of free generators for the Schottky group $G^{\tau}$.

\subsubsection{Vertical projection loops} Next, we proceed to construct a fundamental set of loops for $G^{\tau}$ for the above generators in terms of the fundamental set of loops for $\widehat G$.
In \cite{bm:free} it was shown that, with the above normalization, $f^{\tau}$ converges to the identity $I$, uniformly on compact subsets of $\Omega(\widehat G)$, and, for each $j \in \{1,\ldots,g\}$,  $A_{j}^{\tau}$ converges to ${\widehat A}_{j}$, as $\tau \to \infty$. In particular, if we fix $\tau _0$, and fix $n$, then $f^{\tau }\to I$ uniformly on compact subsets of the truncated flat part of $B^{\tau_0,n}$.  The boundary of $B^{\tau _0,n}$ consists of a disjoint union of quadrilaterals
with circular sides. After renormalization, the part of the boundary of $B^{\tau _0,n}$ corresponding to $\{|\Im(z)|=\tau _0\}$ is the {\em horizontal} part of the boundary, while the part of the boundary corresponding to $\{|\Re(z)|=n\}$ is the {\em vertical} part of the boundary. We make a fixed a choice of the conjugating maps, $H_i$, $i=1,\ldots,q$, and we fix a choice of the parameter $\tau>\tau_0$ in the above construction. 
We may deform all the loops $\widehat \alpha_i$ and $\widehat \alpha'_i$, within  $\Omega^+(\widehat G)$ to an equivalent defining set of loops, with the same geometric generators $\widehat{A}_{1},\ldots,\widehat{A}_{g}$, so that, after appropriate renormalization, each connected  component of each of the deformed loops appears, in each component of the  complement of the flat part of $B^{\tau}$, as a pair of half-infinite  Euclidean vertical lines, one in $\{\Im(z)\ge \tau\} $, the other in $\{\Im(z)\le  -\tau\} $, both with the same real part (the technical details of such a deformation can be found in \cite{H-M}). We call such a deformed loops the {\it vertical projection loops}. These vertical projection loops, which we still denoting as $\widehat \alpha_1,\ldots,\widehat \alpha'_g$,  yields a set of fundamental loops, $\alpha_1^{\tau},\ldots,{\alpha'_{g}}^{\tau}$ for the generators $A_{1}^{\tau},\ldots,A_{g}^{\tau}$ of the Schottky group $G^{\tau}$.

\subsubsection{The relative conical neighborhood}
The {\em relative conical neighborhood} of $\widehat G$ is to be defined as the set of all marked Schottky groups $G^{\tau}=\langle A_{1}^{\tau},\ldots, A_{g}^{\tau}\rangle$, with the fundamental set of loops  $\alpha_1^{\tau},\ldots,{\alpha'_{g}}^{\tau}$, as constructed above.

\begin{rema}
Recall that we are assuming that $\infty$ is an interior point of the flat part corresponding to $\tau_0$, and  $f^\tau (z)=z+O(|z|^{-1})$ near $\infty$. As, with these normalizations, $f^\tau \to I$ uniformly on compact subsets of $\Omega(\widehat G)$, we obtain that  $G^\tau \to \widehat{G}$ algebraically. It now  follows, from the J{\o}rgensen-Marden criterion \cite{J-M:geoconv}, that $G^\tau \to \widehat{G}$ geometrically and that each relative conical neighborhood contains
infinitely many distinct marked Schottky groups. It is also easy to see, as in
\cite{bm:free}, that, for each primitive parabolic element $\widehat
P \in\widehat G$, as $\tau \to\infty$, there is a corresponding geodesic on
$S^\tau =\Omega(G^\tau )/G^\tau $ whose length tends to zero. It follows
that each relative conical neighborhood of a noded Schottky group contains Schottky
groups representing infinitely many distinct Riemann surfaces.
\end{rema}

\subsection{Pinchable loops of Schottky groups}
Let $G$ be a Schottky group of rank $g \geq 2$, with generators $A_{1},\ldots,A_{g}$, and let $\pi:\Omega(G) \to S$ be a regular covering with deck group $G$. 

\subsubsection{Pinchable loops}
Let $\gamma_{1},\ldots,\gamma_{q}$ be a set of simple disjoint geodesics loops on $S$. Each $\gamma_{j}$ corresponds to a conjugacy class of a cyclic subgroup of $G$ (including the trivial subgroup) by the lifting under $\pi$; let $\langle W_{j} \rangle$ be a representing of such a class. If these $q$ cyclic subgroups are non-trivial, they are pairwise non-conjugated in $G$ and the generators $W_{j}$ are non-trivial powers in $G$ (i.e., there is no $T_{j} \in G$ so that $W_{j}=T_{j}^{m_{j}}$ for some $m_{j} \geq 2$), then we say that this set of geodesics is {\em pinchable} in $G$.                                                  

\begin{rema} 
(1) It was shown in \cite{bm:parelt} (see also Yamamoto \cite{yamamoto:parelt})  that if $\gamma_1,\ldots,\gamma_q$ is a set of pinchable simple disjoint geodesics loops on $S$ in $G$, defined by
the words $W_1,\ldots,W_q$, as above, then there is a noded Schottky group $\widehat G$, and there is an isomorphism $\psi:G\to \widehat G$, where $\psi(W_1),\ldots,\psi(W_q)$, and their powers and conjugates, are exactly the parabolic elements of $\widehat G$. More precisely, it was shown in \cite{bm:parelt} that there is a path in Schottky space, ${\mathcal S}_{alg}$, which converges to a set of generators for $\widehat G$, along which the lengths of the geodesics, $\gamma_1,\ldots,\gamma_q$, all tend to zero. 
(2) On the other direction, let us consider a noded Schottky group $\widehat G$ of rank $g \geq 2$, with a set of geometrical generators $\widehat{A}_{1},\ldots, \widehat{A}_{g}$ and corresponding fundamental set of loops $\widehat \alpha_1,\ldots,\widehat \alpha'_g$.
Let $G^{\tau}$ be a Schottky group of rank $g$ with fundamental set of loops  $\alpha_1^{\tau},\ldots,{\alpha'_{g}}^{\tau}$ and generators $A_{1}^{\tau},\ldots, A_{g}^{\tau}$, in a relative conical neighborhoodof $\widehat{G}$ as previously described in Section \ref{sec:vertical}. Let $S=\Omega(G^{\tau})/G^{\tau}$ be the closed Riemann surface of genus $g$ represented by $G^{\tau}$, and let $V_i \subset S$ be the projection of $\alpha_{i}^{\tau}$, $i=1,\ldots,g$. Then $V_1,\ldots,V_g$ is a set of $g$ homologically independent simple disjoint loops on $S$. 
Let $\psi:G^{\tau}\to\widehat G$ be the isomorphism defined by $A_{i}^{\tau}\mapsto \widehat A_i$, $i=1,\ldots,g$. The elements of $G^{\tau}$ which are sent to parabolic elements of $\widehat{G}$ are called the {\it pinched elements} of $G^{\tau}$. There are simple disjoint geodesics $\gamma_1,\ldots,\gamma_q$ on $S$, defined by pinched elements of $G^{\tau}$, given by the words $W_1,\ldots,W_q$ in the generators $A_{1}^{\tau},\ldots, A_{g}^{\tau}$, so that every parabolic element of $\widehat G$ is a power of a conjugate of one of their $\psi$-image. It happens that these collection of loops $\gamma_{1},\ldots,\gamma_{q}$ is a set of pinchable geodesics of $G^{\tau}$. The construction in \cite{bm:free} shows that we can choose the above parameter $\tau$ so that the $\gamma_{i}$ are all arbitrarily short. 
\end{rema}

The following asserts that any set of pinchable geodesics is contained in a maximal one.

\begin{prop}[\cite{H-M}]\label{prop:addpinch} 
Every non-empty set of $k<3g-3$ pinchable geodesics is contained in  a maximal set of $3g-3$ pinchable geodesics.
\end{prop}

\subsubsection{Valid sets of defining loops and their complexity} 
Let  $\gamma_1,\ldots,\gamma_q \subset S$ be a pinchable set of geodesics in $G$.
Set $\widehat S^+$ the  stable Riemann surface obtained from $S$ by pinching these $q$ geodesics;  it consists of a finite number of compact Riemann surfaces,  called {\em parts}, which are joined together at a  finite number of nodes. Also, let $\widehat G$ be the noded Schottky group obtained 
from $G$ by pinching these $q$ geodesics.

Let $V_1,\ldots,V_g$, be a set of defining loops for $G$ on $S$ (that is, the components of the lifting of these loops under $\pi$ are simple loops and such a lifting set of loops contains a fundamental set of loops for $G$) and let $\widehat V_1,\ldots,\widehat V_g$ be the corresponding loops on $\widehat S^+$ obtained by pinching 
$\gamma_1,\ldots,\gamma_q$. We observe that the lifts of the $\widehat V_i$ to  $\Omega^{+}(\widehat{G})$ are all loops, but they are generally not disjoint and they need not to be simple. There are certainly some number of these lifts passing through each parabolic fixed point, and  some of them  might pass more than once through the same parabolic fixed point. The set of  loops, $V_1,\ldots,V_g$, is called a {\em valid set of defining loops} for  $\gamma_1,\ldots,\gamma_q$, if every lift of every $\widehat V_i$ to $\Omega^{+}(\widehat{G})$ is a  simple loop; that is, it passes at most once through each parabolic fixed point (i.e., the set of loops, $\widehat  V_1,\ldots,\widehat V_g$, forms a set of defining loops for $\widehat G$ on $\widehat S^+$).
We note that there are exactly $q$ equivalence classes of parabolic fixed points in $\widehat G$, one for each of the loops $\gamma_i$.

\begin{prop}[\cite{Hidalgo:NodedSchottky,H-M}]\label{existe1}
There is at least  one valid set of defining loops $V_1,\ldots,V_g$, for every set of pinchable  geodesics, $\gamma_1,\ldots,\gamma_q$. 
\end{prop} 

\subsubsection{The complexity} 
Let us now consider a valid set of defining loops, $V_1,\ldots,V_g$, for a set of pinchable geodesics $\gamma_{1},\ldots, \gamma_{q}$. We can deform the $V_i$ on $S$ so that they are all geodesics. Then  the geometric intersection number, $V_i\bullet \gamma_j$, of $V_i$ with $\gamma_j$ is  well defined; it is the number of points of intersection of these two  geodesics. Looking on the corresponding noded surface $\widehat S^+$,  $V_i\bullet \gamma_j$ is the number of times the curve $\widehat V_i$ obtained from $V_i$ by  contracting $\gamma_j$ to a point, passes through that point (node). The {\em complexity} of $V_1,\ldots,V_g$, with respect to $\gamma_{1},\ldots,\gamma_{q}$, is given by 
$$\Xi(\gamma_{1},\ldots,\gamma_{q};V_{1},\ldots,V_{g})=\max_{1 \leq j \leq q}\sum_{i=1}^g V_{i}\bullet \gamma_{j},$$ 
and the {\em complexity} $\Xi(\gamma_1,\ldots,\gamma_q)$ is the minimum of $\Xi(\gamma_{1},\ldots,\gamma_{q};V_{1},\ldots,V_{g})$, 
where the minimum is taken over all valid sets of defining loops.
If $\Xi(\gamma_1,\ldots,\gamma_q)\ge n$, then,  for every valid defining set $V_1,\ldots,V_g$, there is a node  $P$ on $S^+$ so that the total number of crossings of $P$ by $\widehat  V_1,\ldots,\widehat V_g$ is at least $n$.

\begin{prop}[\cite{H-M}]\label{topologico}
Let $g \geq 2$ and $G$ be a Schottky group of rank $g$.
For each positive integers $n$ there are only finitely many topologically distinct maximal (i.e. $q=3g-3$) pinchable set of geodesics in $G$
and complexity $n$. In particular, there are infinitely many topologically distinct maximal noded Schottky groups 
of rank $g$ and there are only finitely many topologically distinct maximal neoclassical noded Schottky groups in each rank $g$. 
\end{prop}

\subsubsection{Sufficiently complicated pinchable sets of geodesics} 
Now, we consider a maximal set of pinchable geodesics in $G$, say 
$\gamma_{1},\ldots,\gamma_{3g-3}$; so $\widehat G$ is a maximal noded Schottky group. Observe that $\widehat G$ is rigid, and that every part of $S^+$ is a sphere with three distinct nodes. Also, every connected component  $\Delta\subset \Omega(\widehat{G})$ is a  Euclidean disc $\Delta$, where $\Delta/\Stab(\Delta)$ is a sphere with three punctures; the three punctures correspond to the three nodes of the corresponding part of $\widehat S^+$. 

Let $V_1,\ldots,V_g$ be a valid set of defining loops on $S$ for the given set of pinchable geodesics (the existence is given by Proposition \ref{existe1}),  and let $\widehat V_1,\ldots,\widehat V_g$ be the corresponding loops on  $\widehat S^+$. For each $i=1,\ldots,g$,  the intersection of a lifting of $\widehat V_i$ with a  component of $\widehat G$ (i.e., a connected component of $\Omega(\widehat{G})$) is  called a {\em strand} of that lifting $\widehat V_i$.  Similarly, the loops $\widehat V_1,\ldots,\widehat V_g$ appear on the corresponding parts of $\widehat S^+$ as collections of {\em strands} connecting the nodes on the boundary of each part. There are two possibilities for these strands; either a strand connects two distinct nodes on some part, or it starts and ends at the same node. Since the loops $V_1,\ldots,V_g$ are simple and disjoint,  there are at most three sets of  {\em parallel} strands of the $\widehat V_i$ in each part; that is, there  are at most three sets of strands, where any two strands in the same set are homotopic  arcs with fixed endpoints at the nodes. We regard each of these sets of strands on a single part as being a {\em superstrand}, so that there are at most $3$ superstrands on any one part.
We next look in some component $\Delta$ of $\Omega(\widehat G)$, and look at a parabolic fixed point $x$ on its boundary, where $x$ corresponds to the node $N$ on the part $S_i$ of $\widehat S^+$. In general, there will be infinitely many liftings of superstrands emanating from $x$ in $\Delta$, but, modulo $\Stab(\Delta)$ there are only finitely many. In fact, there are at most $4$ such liftings of superstrands emanating from $x$. If there is exactly one superstrand on $S_i$ with one endpoint at $N$, and the other endpoint at a different node, then modulo $\Stab(\Delta)$ there will be exactly the one lifting of this superstrand emanating from $x$. If there is only one superstrand on $S_i$ with both endpoints at the same node $N$, then this superstrand has two liftings starting at $x$, one in each direction; so, in this case, we see two lifts of superstrands modulo $\Stab(\Delta)$ emanating from $x$. It follows that, modulo $\Stab(\Delta)$, we can have $0$, $1$, $2$, $3$ or $4$ liftings of superstrands starting at $x$. We note that these liftings of superstrands all end at distinct parabolic fixed points on the boundary of $\Delta$.
We say that the defining set of loops, $\widehat V_1,\ldots,\widehat V_p$ is {\em sufficiently complicated} if there are two (different) lifts $\widehat \alpha_i$ and $\widehat \alpha_j$, of some $\widehat V_i$ and some not necessarily distinct $\widehat V_j$, respectively, so that $\widehat \alpha_i$ and $\widehat \alpha_j$ both pass through the parabolic fixed point $z_1$, into a component $\Delta_1$ of $\widehat G$, then both travel through $\Delta_1$ to the same parabolic fixed point on its boundary, $z_2$, and into another component $\Delta_2$, which they again traverse together to the same boundary point, $z_3$, necessarily a parabolic fixed point, where they enter $\Delta_3$, and they leave $\Delta_{3}$ at different parabolic fixed points.

\subsection{Sufficiently complicated noded Schottky groups}
A maximal noded Schottky group $\widehat G$ is {\em sufficiently complicated} if every set of valid defining loops on $\widehat S^+$ is sufficiently complicated.
We note that (keeping the notation of last section), inside $\Delta_1$, $\widehat \alpha_i$ and $\widehat \alpha_j$ are disjoint; they both enter $\Delta_1$ at the same point, and they both leave $\Delta_1$  at the same point; hence they cannot both be circles. 
In \cite{H-M} the following result, stating a sufficient condition in terms of the complexity for a maximal noded Schottky group to be sufficiently complicated, was obtained.

\begin{theo}[\cite{H-M}]\label{prop:11} 
If a maximal noded Schottky group $\widehat G$ has 
complexity at least $11$, then it is sufficiently complicated. 
\end{theo}

The previous theorem, together Proposition \ref{topologico}, asserts  the existence of infinitely many topologically different sufficiently complicated maximal noded Schottky groups in every rank $g \geq 2$.
The following result states sufficient conditions for a Schottky group to be non-classical.

\begin{theo}[\cite{H-M}] \label{anodado}
Let $\widehat G$ be a maximal noded Schottky group.

\begin{enumerate}
\item If $\widehat G$ is sufficiently complicated, then, for $\tau$ sufficiently large, the Schottky  group $G^{\tau}$ in the relative conical neighborhood of $\widehat G$ is non-classical. 

\item If $S^{+}=\Omega^{+}(\widehat{G})/\widehat{G}$ is the stable Riemann surface as shown in figure \ref{figura1}, then $\widehat{G}$ is sufficiently complicated.
\end{enumerate}
\end{theo}

\section{Explicit construction of a sufficiently complicated noded Schottky group}\label{Sec:construccion}
In this section we construct explicitly a noded Schottky group as in part (2) of Theorem \ref{anodado}, so part (1) of the same theorem asserts that any Schottky group, sufficiently near to $\widehat{G}$, is necessarily non-classical. 

\subsection{A sufficiently complicated noded Schottky groups of rank three}
Let $L_{0}$ be the unit circle, $L_{1}$ be the real line, $L_{2}$ be the line through the points $0$ and $w_{0}=e^{\pi i/3}$, and set
$${\mathcal F}=\left\{(p,r): 0<r<p<1, \; p>1/2, \; r<\frac{\sqrt{1+p^2+p^4}+p^2-1}{\sqrt{3}p} \right\}.$$

For each $(p,r) \in {\mathcal F}$ we set 
$L_{3}$ to be the circle centred at $0$ and radius $r$ and $L_{4}$ to be the circle orthogonal to $L_{0}$,
intersecting $L_{1}$ at the points $p$ and $1/p$ with angle $\pi/3$  (see Figure \ref{figura2}). 
The circle $L_{4}$ is centered at 
$c=\left(\frac{1+p^2}{2p}\right)+\frac{i}{\sqrt{3}}\left(\frac{1-p^2}{2p}\right)$ and 
it has radius $R=\frac{2}{\sqrt{3}}\left(\frac{1-p^2}{2p}\right)$.
The condition $p>1/2$ asserts that $L_{2}$ and $L_{4}$ are disjoint  (tangency occurs when $p=1/2$) and the condition
$r<\frac{\sqrt{1+p^2+p^4}+p^2-1}{\sqrt{3}p}$ asserts that $L_{3}$ and $L_{4}$ are disjoint (tangency occurs when we have an equality). 
Let $\tau_{j}$ be the reflection on $L_{j}$, for $j=0,1,2,3,4$, so
$$\tau_{0}(z)=1/\overline{z}, \quad \tau_{1}(z)=\overline{z}, \quad \tau_{2}(z)=w^2\overline{z}, \quad 
\tau_{3}(z)=r^2/\overline{z}, \quad 
\displaystyle{\tau_{4}(z)=\frac{c\overline{z}-1}{\overline{z}-\overline{c}}},$$
and let $K_{r,p}=\langle \tau_{0}, \tau_{1}, \tau_{2}, \tau_{3}, \tau_{4}\rangle$. It turns out that $K_{r,p}$ is an extended Kleinian group with connected region of discontinuity $\Omega_{r,p}$ and so that 
$\Omega_{r,p}/K_{r,p}$ is a closed disc with $5$ branch values, of orders $2$, $2$, $2$,
$2$ and $3$, on its border. As a consequence of the Klein-Maskit combination theorems \cite{bm:combiv}, the group $K_{r,p}$ has no parabolic transformations, its limit set is a Cantor set and it is geometrically finite.
If we set $W=\tau_{2}\tau_{1}$, $J=\tau_{0}\tau_{1}$ and $L=\tau_{1}\tau_{4}$, then $W^3=L^3=J^2=(WJ)^2=(LJ)^2=1$. Now, if 
$A_{1}=L^{-1}W^{-1}=\tau_{4}\tau_{2}$, $A_{2}=\tau_{1}A_{1}\tau_{1},$ and $A_{3}=\tau_{0}\tau_{3}$, so
$$
A_1(z)=\displaystyle{\frac{c w_{0} z-1}{w_{0} z-\overline{c}}}, \quad 
A_2(z)=\displaystyle{\frac{\overline{c}w_{0}^2z-1}{w_{0}^2z-c}}, \quad 
A_3(z)=r^2z,
$$
then we set $G_{r,p}=\langle A_{1},A_{2},A_{3}\rangle$. In Figure \ref{figura3} we show the situation for values of $p$ near to $1$ and $r$ near to $0$; in which case $G_{r,p}$ turns out to be a classical Schottky group of rank three.

\begin{lemm}
If $(p,r) \in {\mathcal F}$, then $G_{r,p}$ is a Schottky group of rank three.
\end{lemm}
\begin{proof}
It can be seen that $G_{r,p}$ is a finite index normal subgroup of $K_{r,p}$ and
$$K_{r,p}/G_{r,p}=\langle \tau_{0} : \tau_{0}^2=1\rangle \times \langle \tau_{1},\tau_{2} :
\tau_{1}^2=\tau_{2}^2=(\tau_{2}\tau_{1})^3=1\rangle \cong {\mathbb Z}_{2} \oplus D_{3}.
$$

In particular, $G_{r,p}$ has the same region of discontinuity as for $K_{r,p}$ (so a function group), it is geometrically finite and  does not have parabolic transformations. As any of the elliptic transformations of $K_{r,p}$ goes into an element of the same order in the quotient $K_{r,p}/G_{r,p}$, we also have that $G_{r,p}$ is torsion free. Now, as a consequence of the classification of function groups \cite{Maskit:function3,Maskit:function4}, the group $G_{r,p}$ is a Schottky group of rank three (in Figure \ref{figura3} there is shown a set of fundamental loops). 
\end{proof}

The closed Riemann surface $S_{r,p}=\Omega_{r,p}/G_{r,p}$ of genus $3$ admits the group ${\mathbb Z}_{2} \oplus D_{3}$ as a group of conformal/anticonformal automorphisms. On $S_{r,p}$ we have simple closed curves
$\gamma_{1}$,..., $\gamma_{6}$ which are pinchable (see Figures \ref{figura4} and \ref{figura5}) with respect to the Schottky group $G_{r,p}$; these pinchable curves correspond to the conjugacy classes of cyclic groups of $G_{r,p}$ as follows:
$$
\begin{array}{l}
\mbox{$\gamma_{1}$ corresponds to $A_{2}^{-1}$; $\gamma_{2}$ corresponds to $A_{1}$; $\gamma_{3}$ corresponds to $A_{1}^{-1}A_{2}$;}\\
\mbox{$\gamma_{4}$ corresponds to $A_{1}^{-1}A_{2}A_{3}^{-1}A_{2}^{-1}A_{1}A_{3}$; $\gamma_{5}$ corresponds to $A_{2}^{-1}A_{3}^{-1}A_{2}A_{3}$;}\\
\mbox{$\gamma_{6}$ corresponds to $A_{1}A_{3}^{-1}A_{1}^{-1}A_{3}$.}
\end{array}
$$

\begin{figure}[!tbp]
  \centering
  \begin{minipage}[b]{0.49\textwidth}
  \centering
    \includegraphics[width=4cm]{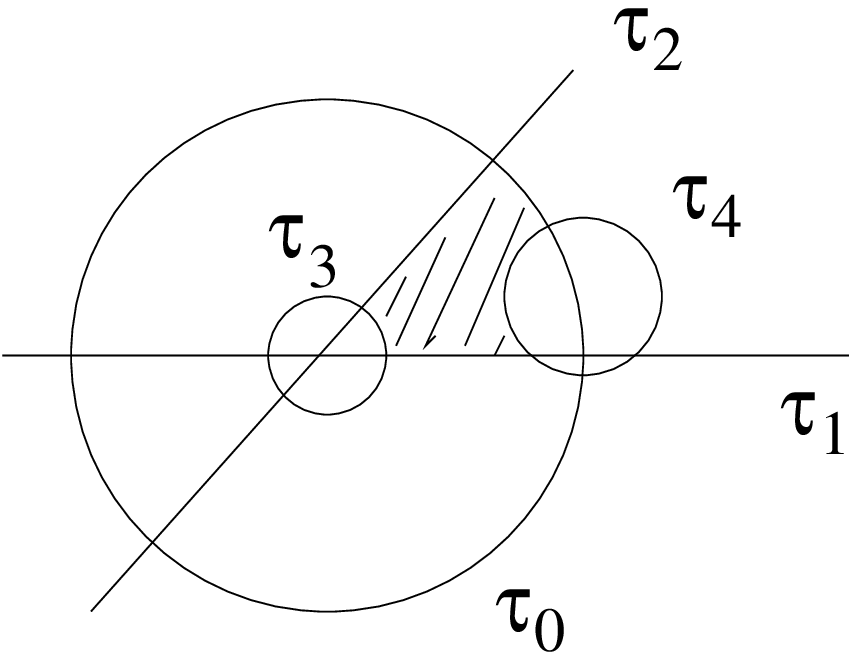}
    \caption{A set of lines and circles}
    \label{figura2}
  \end{minipage}
  \hfill
  \begin{minipage}[b]{0.49\textwidth}
  \centering
    \includegraphics[width=4cm]{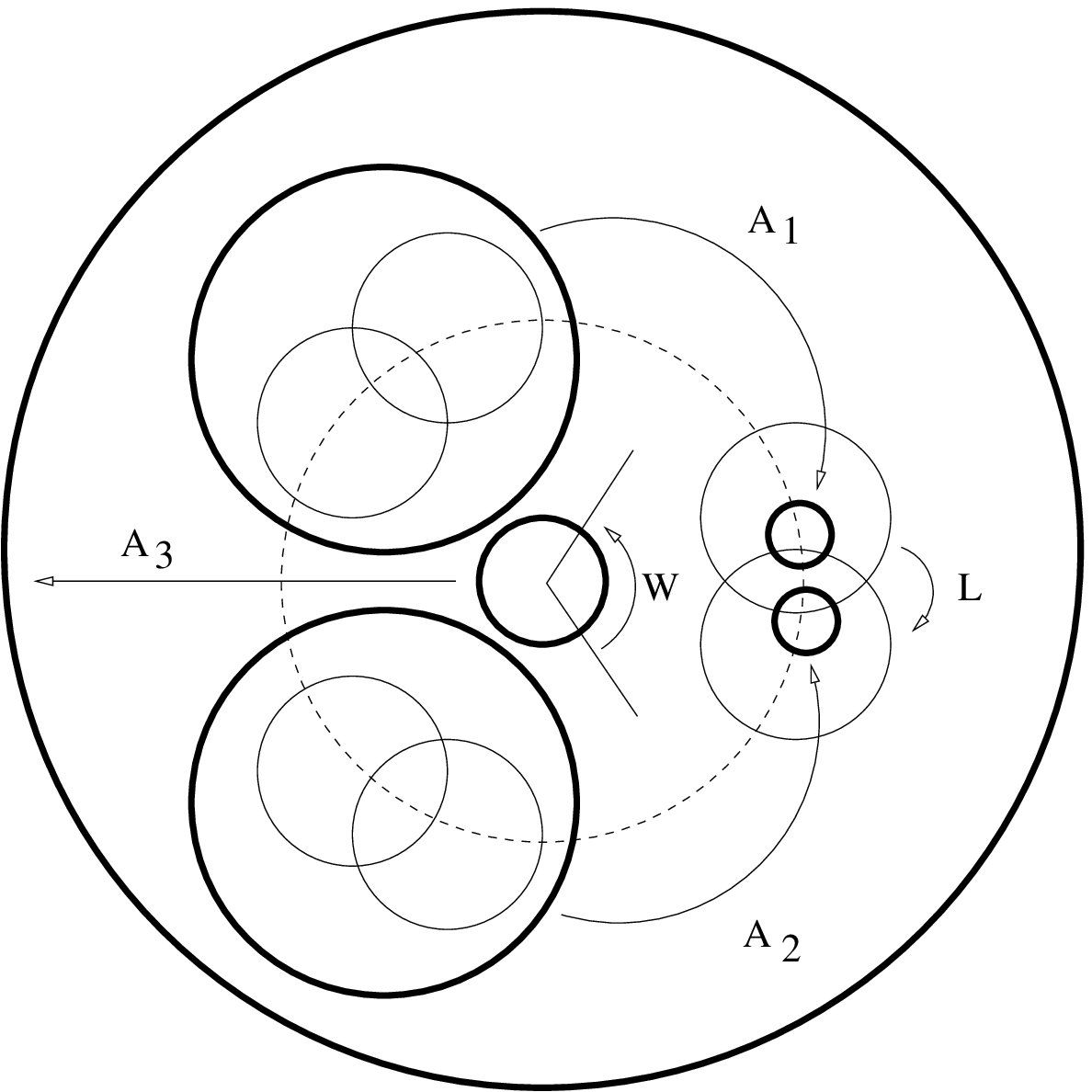}
    \caption{The Schottky group $G_{r,p}$ of rank $3$: the six darkest loops are a set of fundamental loops}
    \label{figura3}
  \end{minipage}
\end{figure}

\begin{figure}[!tbp]
  \centering
  \begin{minipage}[b]{0.48\textwidth}
  \centering
    \includegraphics[width=4.8cm]{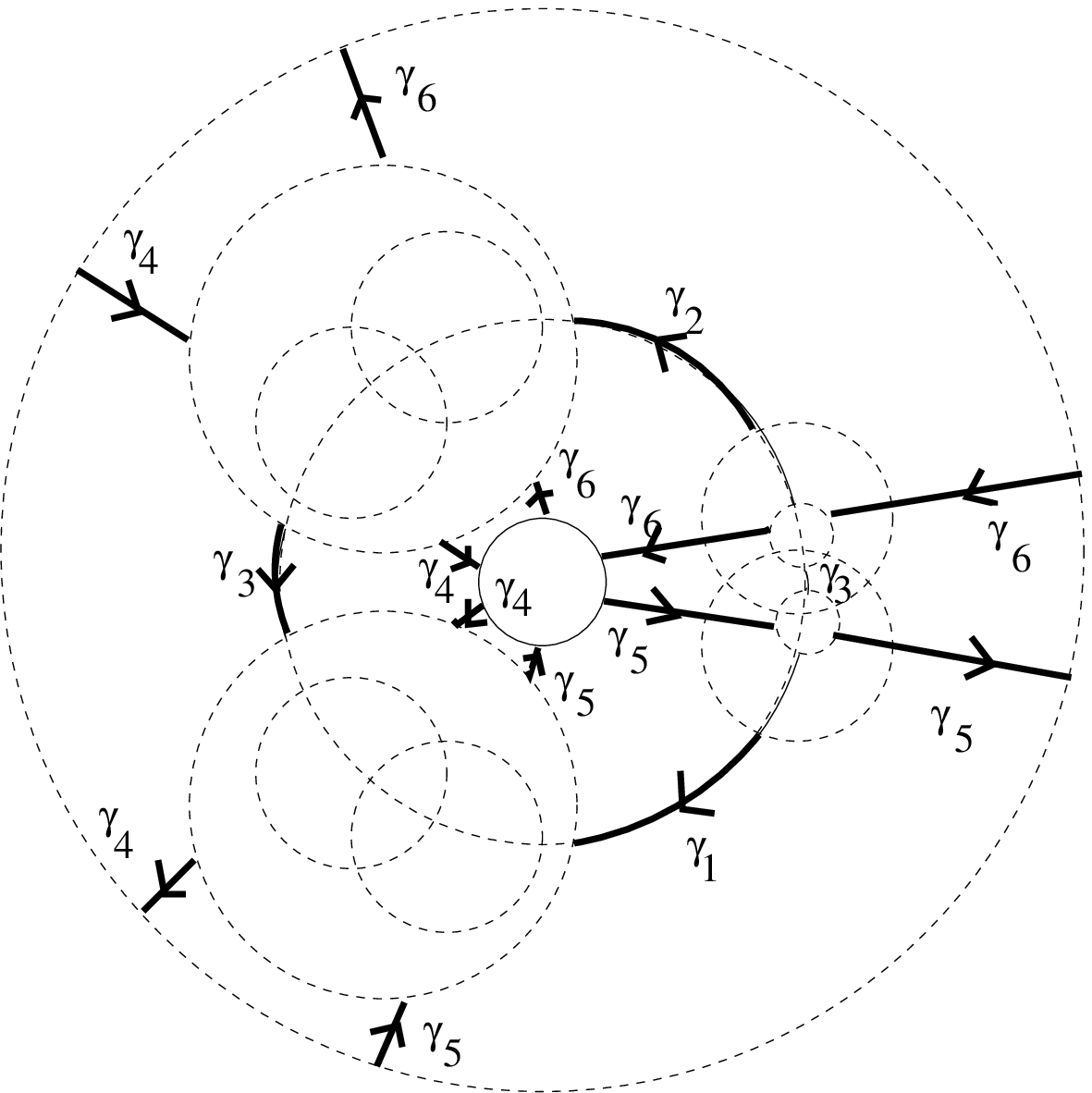}
    \caption{A set of pinchable curves seen at the Schottky uniformization}
    \label{figura4}
  \end{minipage}
  \hfill
  \begin{minipage}[b]{0.48\textwidth}
  \centering
    \includegraphics[width=4cm]{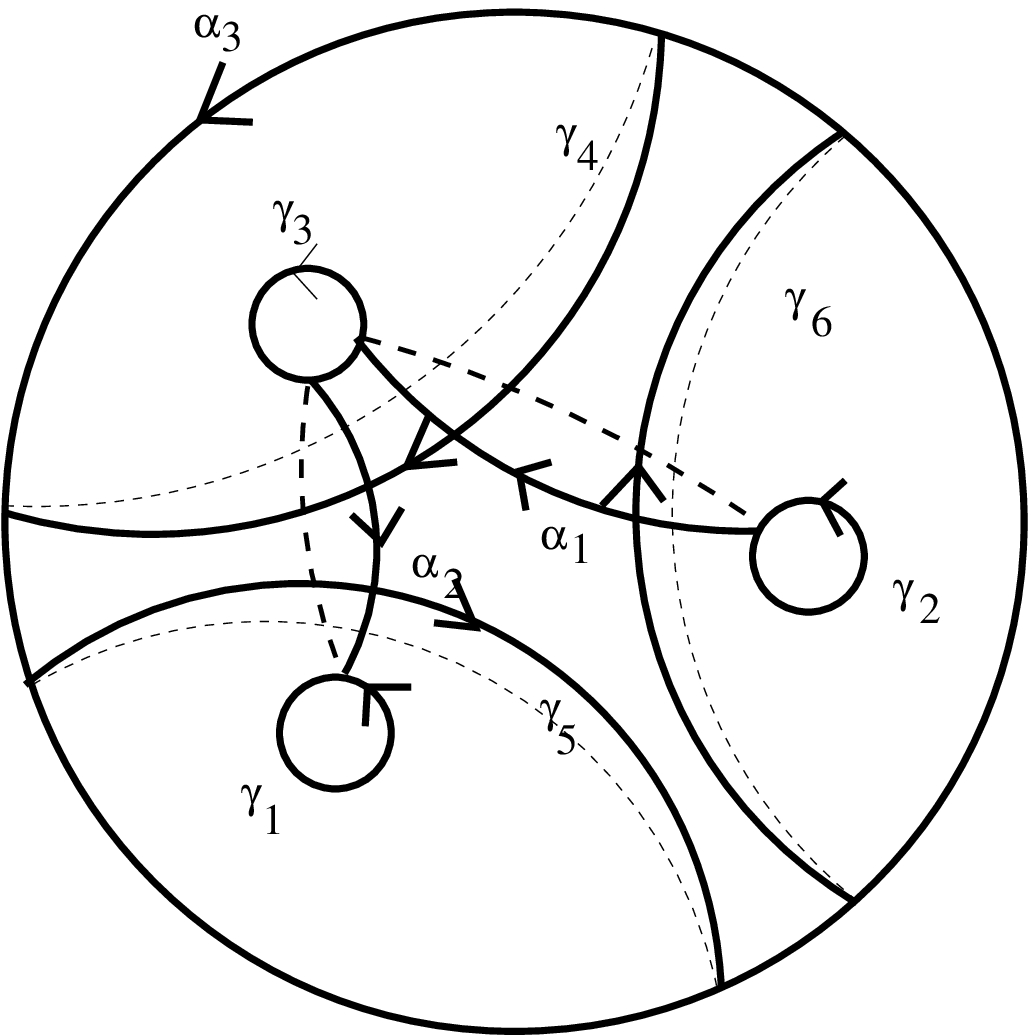}
    \caption{A set of pinchable curves seen on the Riemann surface $S_{r,p}$}
    \label{figura5}
  \end{minipage}
\end{figure}

To obtain the desired noded Schottky group, we need to move the parameters $p$ and $r$ in order to have that the loxodromic transformations 
$A_{2}^{-1}$, $A_{1}$, $A_{1}^{-1}A_{2}$, $A_{1}^{-1}A_{2}A_{3}^{-1}A_{2}^{-1}A_{1}A_{3}$, $A_{2}^{-1}A_{3}^{-1}A_{2}A_{3}$ and $A_{1}A_{3}^{-1}A_{1}^{-1}A_{3}$ are transformed into parabolic transformations.
As the order three automorphism $W$ permutes cyclically $\gamma_{1},\gamma_{2},\gamma_{3}$ and also
$\gamma_{4},\gamma_{5},\gamma_{6}$, we only need to take care of $A_{1}$ and $A_{1}A_{3}^{-1}A_{1}^{-1}A_{3}$.

First, in order to make $A_{1}$ a parabolic we only need to have $\tau_{4}(w_{0})=w_{0}$, equivalently, that the circle $L_{4}$ is
tangent to the line $L_{2}$ at $w_{0}$. This happens exactly when $p=1/2$. Now, assuming $p=1/2$, in order to make $A_{1}A_{3}^{-1}A_{1}^{-1}A_{3}$ parabolic we only need to have tangency
of the circle $L_{3}$ with $L_{4}$, that is,
$r=r^{*}=\frac{\sqrt{7}-\sqrt{3}}{2}$.
The group $G_{r^{*},1/2}$ turns out to be a noded Schottky group that uniformizes a stable Riemann surface as shown in Figure \ref{figura1} and, by (2) in Theorem \ref{anodado}, it is a sufficiently complicated noded Schottky group. 

\subsection{Non-classical Schottky groups of rank three}
By (1) in Theorem \ref{anodado},  there exist $(p_{0},r_{0}) \in {\mathcal F}$  with the property that if $(p,r) \in {\mathcal F}$, 
$1/2<p<p_{0}$ and $r_{0}<r<r^{*}$, then $G_{r,p}$ is a non-classical Schottky group of rank three. 
Moreover, each of these Schottky groups is contained in a Kleinian group $K_{r,p}$ as a finite index normal subgroup with
$K_{r,p}/G_{r,p} \cong {\mathbb Z}_{2} \oplus D_{3}$, in other words, the closed Riemann surfaces $S_{r,p}=\Omega(G_{r,p})/G_{r,p}$ admit a group of conformal automorphisms isomorphic to ${\mathbb Z}_{2} \oplus D_{3}$. The family of these Riemann surfaces degenerates to a stable Riemann surface $S_{r^{*},1/2}$ as Figure \ref{figura1} keeping the above group of automorphisms invariant.



\begin{thebibliography}{99}

\bibitem{Bobenko}
A. I. Bobenko.
Schottky uniformization and finite-gap integration,
{\it Soviet Math. Dokl.} {\bf 36} No. 1 (1988), 38--42 (transl. from Russian:
{\it Dokl. Akad. Nauk SSSR}, {\bf 295}, No.2 (1987)).


\bibitem{Button}
J. Button.
All Fuchsian Schottky groups are classical Schottky groups.
{\it Geometry \& Topology Monographs}.
Volume 1: The Epstein birthday schrift (1998), 117--125.

\bibitem{Chuckrow}
V. Chuckrow.
Schottky groups and limits of Kleinian groups.
{\it Bull. Amer. Math. Soc.} {\bf 73} No. 1 (1967), 139--141.

\bibitem{Hidalgo:NodedSchottky}
R. A. Hidalgo.
The noded {S}chottky space.
{\it London Math. Soc.} {\bf 73} (1996), 385--403.

\bibitem{Hidalgo:NodedFuchsian}
R. A . Hidalgo.
Noded Fuchsian groups.
{\it Complex Variables} {\bf 36} (1998), 45--66.

\bibitem{Hidalgo:classical}
R. A. Hidalgo.
Towards a proof of the classical Schottky uniformization conjecture.
https://arxiv.org/pdf/1709.09515.pdf


\bibitem{H-M}
R. A. Hidalgo and B. Maskit.
On neoclassical Schottky groups.
{\it Trans. of the Amer. Math. Soc.} {\bf 358} (2006), 4765--4792.


\bibitem{Hou1}
Y. Hou. 
On smooth moduli space of Riemann surfaces.
(2016).\\ https://arxiv.org/pdf/1610.03132.pdf

\bibitem{Hou2}
Y. Hou. 
The classification of Kleinian groups of Hausdorff dimensions at most one.
(2016). \\  https://arxiv.org/pdf/1610.03046.pdf



\bibitem{J-M:geoconv}
T. J{\o}rgensen and A. Marden.
Algebraic and geometric convergence of Kleinian groups.
{\it Math. Scand.} {\bf 66} (1990), 47--72.

\bibitem{Koebe}  
P. Koebe.  
\"{U}ber die Uniformisierung der Algebraischen Kurven II.
{\it Math. Ann.} {\bf 69} (1910), 1--81.



\bibitem{K-M:pinched}
I. Kra and B. Maskit.
Pinched two component Kleinian groups.
In {\it Analysis and Topology}, pages 425--465. World Scientific Press, 1998.

\bibitem{McMullen}
C. McMullen.
{\it Complex Dynamics and Renormalization}. 
Annals of Mathematical Studies {\bf 135}, Princeton University Press, (1984).




\bibitem{Marden}
A. Marden.
Schottky groups and circles.
In {\it Contributions to Analysis (a collection of papers dedicated to Lipman Bers)}, 273--278, Academic Press, 1974.

\bibitem{Maskit:function3}
B. Maskit. 
On the classification of Kleinian Groups I. Koebe groups. 
{\it Acta Math.} {\bf 135} (1975), 249--270.


\bibitem{Maskit:function4}
B. Maskit. 
On the classification of Kleinian Groups II. Signatures. 
{\it Acta Math.} {\bf 138} (1976), 17--42.

\bibitem{Maskit:Schottky}
B. Maskit.
A characterization of Schottky groups.
{\it J. Analyse Math. } {\bf 19} (1967), 227--230.


\bibitem{Maskit:conj}
B. Maskit.
Remarks on m-symmetric Riemann surfaces.
{\it Contemporary Math.} {\bf 211} (1997), 433--445.

\bibitem{bm:combiv}
B. Maskit.
On Klein's combination theorem IV.
{\it Trans. Amer. Math. Soc.} {\bf 336} (1993),265--294.

\bibitem{bm:free}
B. Maskit.
On free Kleinian groups.
{\it Duke Math. J.} {\bf 48} (1981),755--765.

\bibitem{bm:parelt}
B. Maskit.
Parabolic elements in Kleinian groups.
{\it Annals of Math.} {\bf 117} (1983), 659--668.

\bibitem{P}
N. Purzitsky.
Two-Generator  Discrete  Free  Products.
{\it Math. Z.} {\bf 126} (1972), 209--223.

\bibitem{Sato}
H. Sato.
On a paper of Zarrow.
{\it Duke Math. J.} {\bf 57} (1988), 205--209.


\bibitem{Seppala}
M. Sepp\"al\"a.
Myrberg's numerical uniformization of hyperelliptic curves.
{\it Ann. Acad. Scie. Fenn. Math.} {\bf 29} (2004), 3--20.


\bibitem{Williams}
J. P. Williams.
Classical and non-classical Schottky groups.
Doctoral Thesis, University of Southampton, School of Mathematics (2009), 125pp.

\bibitem{yamamoto:parelt}
Hiro-o Yamamoto.
Sqeezing deformations in Schottky spaces.
{\it J. Math. Soc. Japan} {\bf 31} (1979), 227--243.


\bibitem{Yamamoto}
Hiro-o Yamamoto.
An example of a non-classical Schottky group.
{\it Duke Math. J.} {\bf 63} (1991), 193--197.

\bibitem{Zarrow}
R. Zarrow.
Classical and non-classical Schottky groups.
{\it Duke Math. J.} {\bf 42} (1975), 717--724.

\end{thebibliography}
\end{document}